\documentclass[11pt,reqno]{amsart}

\usepackage[utf8]{inputenc}

\usepackage{amssymb}
\usepackage{amsmath}
\usepackage{amsthm}
\usepackage{latexsym}
\usepackage{amsfonts}
\usepackage{color}
\usepackage{url}
\usepackage{todonotes}
\usepackage[unicode=true]{hyperref}
\usepackage{cleveref}

\newtheorem{thm}{Theorem}[section]
\newtheorem{cor}[thm]{Corollary}
\newtheorem{lem}[thm]{Lemma}
\newtheorem{prop}[thm]{Proposition}

\theoremstyle{definition}

\usepackage{thmtools}
\usepackage{thm-restate}

\numberwithin{equation}{section}
\allowdisplaybreaks

\newcommand{\braces}[1]{{\rm (}#1{\rm )}}

\newcommand{\Id}{\operatorname{Id}}

\usepackage{xcolor}
\usepackage[normalem]{ulem}  

\newcommand{\Fourier}{\mathcal{F}}

\newcommand{\WeightedBessel}{L_{w,\Lambda}^2(\R^d)}
\newcommand{\NormWeightedBessel}[1]{\| #1 \|_{L_{w,\Lambda}^2}}

\newcommand{\R}{\ensuremath{\mathbb R}}    
\newcommand{\CC}{\ensuremath{\mathbb C}}   
\newcommand{\N}{\ensuremath{\mathbb N}}    
\newcommand{\Z}{\ensuremath{\mathbb Z}}    

\newcommand{\iso}{\circ}

\newcommand{\<}{\langle}
\renewcommand{\>}{\rangle}

\newcommand{\calB}{\mathcal B}

\newcommand{\calF}{\mathcal F}         
         
\newcommand{\calH}{\mathcal H}

\newcommand{\calS}{\mathcal S}

\newcommand{\la}{\lambda}
\newcommand{\veps}{\varepsilon}
\newcommand{\eps}{\varepsilon}

\newcommand{\mat}[4]
{
   \begin{pmatrix}
      #1 & #2\\
      #3 & #4
   \end{pmatrix}
}

\renewcommand{\Im}{\operatorname{Im}}
\renewcommand{\Re}{\operatorname{Re}}
\newcommand{\linspan}{\operatorname{span}}

\newcommand{\ran}{\operatorname{ran}}

\newcommand{\sap}{\sigma_{{ap}}}

\newcommand{\downto}{\downarrow}

\newcommand{\ol}{\overline}

\newcommand{\wh}{\widehat}

\newcommand{\supp}{\operatorname{supp}}

\newcommand{\La}{\Lambda}

\newcommand{\HH}{\mathbb{H}}

\newcommand{\del}{\partial}
\newcommand{\topp}{{\!\top}}
\newcommand{\Vol}{\operatorname{Vol}}

\binoppenalty=\maxdimen
\relpenalty=\maxdimen
\begin{document}

\title[Invertibility of the Gabor frame operator on certain modulation spaces]
      {A note on the invertibility of the Gabor frame operator on certain modulation spaces}


\author[D.G. Lee]{Dae Gwan Lee}
\address{{\bf D.G.~Lee:} KU Eichst\"att--Ingolstadt, Mathematisch--Geographische Fakult\"at,
Os\-ten\-stra\-\ss{}e~26, Kollegiengeb\"aude I Bau B, 85072 Eichst\"att, Germany}
\email{daegwans@gmail.com}

\author[F. Philipp]{Friedrich Philipp}
\address{{\bf F.~Philipp:} Technische Universit\"at Ilmenau, Institute for Mathematics,
Weimarer Stra\ss e 25, D-98693 Ilmenau, Germany}
\email{friedrich.philipp@tu-ilmenau.de}

\author[F. Voigtlaender]{Felix Voigtlaender}
\address{{\bf F.~Voigtlaender:}
Department of Mathematics,
Technical University of Munich,
85748 Garching bei München,
Germany}
\email{felix@voigtlaender.xyz}


\begin{abstract}
  We consider Gabor frames generated by a general lattice and a window function 
  that belongs to one of the following spaces:
  the Sobolev space $V_1 = H^1(\R^d)$,
  the weighted $L^2$-space $V_2 = L_{1 + |x|}^2(\R^d)$,
  and the space $V_3 = \HH^1(\R^d) = V_1 \cap V_2$
  consisting of all functions with finite uncertainty product;
  all these spaces can be described as modulation spaces with respect to suitable weighted
  $L^2$ spaces.
  In all cases, we prove that the space of Bessel vectors in $V_j$ is mapped bijectively onto itself
  by the Gabor frame operator. 
  As a consequence, if the window function belongs to one of the three spaces,
  then the canonical dual window also belongs to the same space.
  In fact, the result not only applies to frames, but also to frame sequences.
\end{abstract}


\subjclass[2010]{Primary: 42C15. Secondary: 42C40, 46E35, 46B15.}
\keywords{Gabor frames, Sobolev space, Invariance, Dual frame, Regularity of dual window.}

\maketitle
\thispagestyle{empty}

\section{Introduction}

Analyzing the time-frequency localization of functions
is an important topic in harmonic analysis.
Quantitative results on this localization are usually formulated in terms of function spaces
such as Sobolev spaces, modulation spaces, or Wiener amalgam spaces.
An especially important space is the \emph{Feichtinger algebra} $S_0 = M^1$
\cite{FeichtingerNewSegalAlgebra,JakobsenNoLongerNewSegalAlgebra} which has numerous
remarkable properties; see, e.g., \cite[Section A.6]{ChristensenBook} for a compact overview.
Yet, in some cases it is preferable to work with more classical spaces like the Sobolev space
$H^1 (\R^d) = W^{1,2}(\R^d)$, 
the weighted $L^2$-space $L_{1 + |x|}^2(\R^d) = \{ f : \R^d \to \CC : (1 + |x|) f(x) \in L^2 \}$,
or the space $\HH^1(\R^d) = H^1 (\R^d) \cap L_{1 + |x|}^2(\R^d)$ which consists of all functions
$g \in L^2(\R^d)$ with finite uncertainty product
\begin{equation}\label{eqn:FUP}
  \left(
    \int_{\R^d}
      |x|^2
      \cdot |g(x)|^2
    \, dx
  \right)
  \left(
    \int_{\R^d}
      |\omega|^2
      \cdot |\wh g(\omega)|^2
    \,d \omega
  \right)
  < \infty .
\end{equation}
Certainly, one advantage of these classical spaces is that membership of a function in the space
can be decided easily. 
We remark that all of these spaces fall into the scale of modulation spaces
(see Section~\ref{s:Invariance}).

In Gabor analysis, it is known (see e.g., \cite[Proposition~5.2.1]{GroechenigTFFoundations} and \cite[Theorem~12.3.2]{ChristensenBook}) that for a Gabor frame generated by a lattice, the canonical dual frame is again a Gabor system (over the same lattice), generated by the so-called \emph{dual window}.
An important question is what kind of time-frequency localization conditions
are inherited by the dual window.
Precisely, if $g \in L^2(\R^d)$ belongs to a certain ``localization Banach space'' $V$
and if $\La \subset \R^{2d}$ is such that $(g,\Lambda)$ forms a Gabor frame for $L^2(\R^d)$,
then does the canonical dual window belong to $V$ as well?
A celebrated result in time-frequency analysis states that this is true for the
Feichtinger algebra $V = S_0(\R^d)$; see \cite{GroechenigLeinert} for separable lattices $\Lambda$
and \cite[Theorem~7]{BalanDensityOvercompletenessLocalization} for irregular sets $\Lambda$.
In the case of separable lattices, the question has been answered affirmatively
also for the Schwartz space $V = \calS(\R)$ \cite[Proposition~5.5]{j}
and for the Wiener amalgam space $V = W(L^{\infty},\ell_v^1)$
with a so-called \emph{admissible weight} $v$; see \cite{ko}.
Similarly, the setting of the spaces $V = W(C_\alpha,\ell_v^q)$ (with the H\"{o}lder spaces $C_\alpha$)
is studied in \cite{w}---but except in the case $q = 1$, some additional assumptions
on the window function $g$ are imposed.

To the best of our knowledge, the question has not been answered for modulation spaces other than
$V = M^1_v$, and in particular, not for any of the spaces $V = H^1(\R^d)$, $V = L_{1 + |x|}^2(\R^d)$,
and $V = \HH^1(\R^d)$.
In this note, we show that the answer is affirmative for all of these spaces:

\begin{thm}\label{t:main}
  Let $V \in \{ H^1(\R^d), L_{1 + |x|}^2(\R^d), \HH^1(\R^d) \}$.
  Let $g \in V$ and let $\La \subset \R^{2d}$ be a lattice such that
  the Gabor system $(g, \Lambda)$ is a frame for $L^2(\R^d)$ with frame operator $S$.
  Then the canonical dual window $S^{-1} g$ belongs to $V$.
  Furthermore, $(S^{-1/2} g, \Lambda)$ is a Parseval frame for $L^2(\R^d)$ with $S^{-1/2} g \in V$.
\end{thm}

As mentioned above, the corresponding statement of Theorem~\ref{t:main} for $V = S_0(\R^d)$
with separable lattices $\Lambda$ was proved in \cite{GroechenigLeinert}.
In addition to several deeper insights, the proof given in \cite{GroechenigLeinert}
relies on a simple but essential argument showing that the frame operator $S = S_{\La,g}$
maps $V$ boundedly into itself, which is shown in \cite{GroechenigLeinert}
based on Janssen's representation of $S_{\La,g}$. 
In our setting, this argument is not applicable, because---unlike in the case of $V = S_0(\R^d)$---%
there exist functions $g \in \HH^1$ for which $(g,\Lambda)$ is not an $L^2$-Bessel system. 
In addition, the series in Janssen's representation is not even guaranteed
to converge unconditionally in the strong sense for $\HH^1$-functions,
even if $(g,\Lambda)$ is an $L^2$-Bessel system; see Proposition~\ref{prop:MainResult}.
To bypass these obstacles, we introduce for each space $V \in \{ H^1, L_{1 + |x|}^2, \HH^1 \}$
the associated subspace $V_\Lambda$ consisting of all those functions $g \in V$
that generate a Bessel system over the given lattice $\La$.

We remark that most of the existing works concerning the regularity of the (canonical) dual window
rely on deep results related to Wiener's $1/f$-lemma
on absolutely convergent Fourier series.
In contrast, our methods are based on elementary spectral theory
(see \Cref{s:spectra}) and on certain observations regarding the interaction
of the Gabor frame operator with partial derivatives; see Proposition~\ref{p:bounded}.

The paper is organized as follows:
Section~\ref{s:BesselVectors} discusses the concept of Gabor Bessel vectors
and introduces some related notions.
Then, in Section~\ref{s:Invariance}, we endow the space $V_\Lambda$
(for each choice $V \in \{ H^1, L_{1 + |x|}^2, \HH^1 \}$)
with a Banach space norm and show that the frame operator $S$
maps $V_\Lambda$ boundedly into itself, provided that
the Gabor system $(g,\Lambda)$ is an $L^2$-Bessel system and that
the window function $g$ belongs to $V$.
Finally, we prove in Section~\ref{s:spectra} that for any $V \in \{ H^1, L_{1 + |x|}^2, \HH^1 \}$
the spectrum of $S$ as an operator on $V$ coincides with the spectrum of $S$ as an operator on $L^2$.
This easily implies our main result, Theorem~\ref{t:main}.

\section{Bessel vectors}
\label{s:BesselVectors}

For $a,b\in\R^d$ and $f\in L^2(\R^d)$ we define the operators of translation
by $a$ and modulation by $b$ as
\[
  T_a f(x) := f(x-a)
  \quad \text{and} \quad
  M_b f(x) := e^{2\pi ib\cdot x} \cdot f(x),
\]
respectively.
Both $T_a$ and $M_b$ are unitary operators on $L^2(\R^d)$
and hence so is the {\em time-frequency shift}
\[
  \pi(a,b)
  := T_a M_b
   = e^{-2\pi ia\cdot b} \, M_b T_a .
\]
The Fourier transform $\calF$ is defined on $L^1(\R^d) \cap L^2(\R^d)$ by
$\Fourier f (\xi) = \widehat{f}(\xi) = \int_{\R^d} f(x) e^{-2 \pi i x \cdot \xi} \, d x$
and extended to a unitary operator on $L^2(\R^d)$.
For $z = (z_1, z_2) \in \R^d \times \R^d \cong \R^{2d}$ and $f \in L^2(\R^d)$,
a direct calculation shows that
\begin{equation}\label{e:FTpi}
  \calF [\pi(z)f] = e^{-2\pi iz_1\cdot z_2}\cdot\pi(Jz)\wh f,
\end{equation}
where
\[
  J = \mat 0{I}{-I}0.
\]
A (full rank) {\em lattice} in $\R^{2d}$ is a set of the form $\La = A\Z^{2d}$,
where $A\in\R^{2d\times 2d}$ is invertible.
The volume of $\La$ is defined by $\Vol(\La) := |\!\det A|$
and its density by $d(\La) := \Vol(\La)^{-1}$.
The {\em adjoint lattice} of $\La$ is denoted and defined by $\La^\iso := JA^{-\topp}\Z^{2d}$.

The Gabor system generated by a window function $g\in L^2(\R^d)$
and a lattice $\La\subset\R^{2d}$ is given by
\[
  (g,\Lambda) := \bigl\{ \pi(\la)g : \la\in\Lambda \bigr\}.
\]
We say that $g\in L^2(\R^d)$ is a {\em Bessel vector} with respect to $\La$
if the system $(g,\La)$ is a Bessel system in $L^2(\R^d)$,
meaning that the associated {\em analysis operator} $C_{\La,g}$ defined by 
\begin{equation}\label{eq:CoefficientOperator}
  C_{\La,g} f := \big(\<f,\pi(\la)g \>\big)_{\la\in\La},
  \qquad f \in L^2(\R^d) ,
\end{equation}
is a bounded operator from $L^2(\R^d)$ to $\ell^2(\La)$.
We define
\[
  \calB_\La
  := \big\{g\in L^2(\R^d) : (g,\La)\text{ is a Bessel system}\big\} ,
\]
which is a dense linear subspace of $L^2(\R^d)$ because each Schwartz function
is a Bessel vector with respect to any lattice; see \cite[Theorem~3.3.1]{fz}.
It is well-known that $\calB_\La = \calB_{\La^\iso}$ 
(see, e.g., \cite[Proposition~3.5.10]{fz}).
In fact, we have for $g \in \calB_\La$ that
\begin{equation}\label{e:Cnorms}
  \big\| C_{\La^\iso,g} \big\|
  = \Vol(\La)^{1/2} \cdot \big\| C_{\La,g} \big\| ;
\end{equation}
see \mbox{\cite[proof of Theorem~2.3.1]{k}}.
The {\em cross frame operator} $S_{\La,g,h}$ with respect to $\La$
and two functions $g,h\in\calB_\La$ is defined by
\[
  S_{\La,g,h} := C_{\La,h}^*C_{\La,g}.
\]
In particular, we write $S_{\La,g} := S_{\La,g,g}$ which is called
the {\em frame operator} of $(g,\La)$.
The system $(g,\La)$ is called a \emph{frame} if $S_{\La,g}$ is bounded and boundedly invertible
on $L^2(\R^d)$, that is, if $A \Id_{L^2(\R^d)} \leq S_{\La,g} \leq B \Id_{L^2(\R^d)}$
for some constants $0 < A \leq B < \infty$ (called the frame bounds). 
In particular, a frame with frame bounds $A=B=1$ is called a \emph{Parseval frame}.

In our proofs, the so-called {\em fundamental identity of Gabor analysis}
will play an essential role.
This identity states that
\begin{equation}\label{e:fi}
  \sum_{\la\in\La}\<f,\pi(\la)g\>\<\pi(\la)\gamma,h\>
  = d(\La) \cdot \sum_{\mu\in\La^\iso}
                   \<\gamma,\pi(\mu)g\>
                   \<\pi(\mu)f,h\>.
\end{equation}
It holds, for example, if $f, h \in M^1(\R^d) = S_0(\R^d)$ (the Feichtinger algebra)
and $g,\gamma\in L^2(\R^d)$; see \mbox{\cite[Theorem~3.5.11]{fz}}.
We will use the following version of the fundamental identity:

\begin{lem}
  The fundamental identity \eqref{e:fi} holds if $g,h\in\calB_\La$ or $f,\gamma\in\calB_\La$.
\end{lem}

\begin{proof}
In \cite[Subsection~1.4.1]{j2}, the claim is shown for separable lattices in $\R^2$.
Here, we provide a short proof for the general case.
If $g,h \in \calB_{\Lambda}$, then Equation~\eqref{e:Cnorms}
shows that both sides of Equation~\eqref{e:fi} depend continuously on $f,\gamma \in L^2$.
Similarly, if $f,\gamma \in \calB_{\Lambda}$ then both sides of Equation~\eqref{e:fi}
depend continuously on $g,h \in L^2$.
Therefore, and because $\calB_\La$ is dense in $L^2(\R^d)$,
it is no restriction to assume that $f, g, h, \gamma \in \calB_\La$.
Let $\La = A\Z^{2d}$ and define the function
\[
  G(x)
  = \sum_{n\in\Z^{2d}}
      \big\<f,\pi(A(n-x))g\big\>
      \big\<\pi(A(n-x))\gamma,h\big\> ,
  \quad x\in\R^{2d}.
\]
Writing $A x = ( (A x)_1, (A x)_2 ) \in \R^d \times \R^d$, a direct computation shows that
\[
  \langle f, \pi(A(n-x)) g \rangle
  = e^{2 \pi i (A x)_2 \cdot ( (A x)_1 - (A n)_1)}
    \cdot \langle \pi (A x) f, \pi(A n) g \rangle .
\]
Therefore, and because of $g, \gamma \in \calB_\La$ and since $z \mapsto \pi(z) u$
is continuous on $\R^{2d}$ for each $u \in L^2(\R^d)$, the function $G$ is continuous.
Furthermore, we have
\begin{align*}
  \sum_{n \in \Z^{2d}}
    |\langle f, \pi(A(n-x)) g \rangle
    \cdot \langle \pi( A(n-x)) \gamma, h \rangle|
  & = \sum_{n \in \Z^{2 d}}
        |\langle \pi(A x) f, \pi(A n) g \rangle|
        \cdot |\langle \pi(A n) \gamma, \pi(A x) h \rangle| \\
  & \leq \| C_{\Lambda,g} [\pi(A x) f] \|_{\ell^2}
         \cdot \| C_{\Lambda,\gamma} [\pi(A x) h] \|_{\ell^2} \\
  & \leq \| C_{\Lambda,g} \|
         \cdot \| C_{\Lambda,\gamma} \|
         \cdot \| \pi(A x) f \|_{L^2}
         \cdot \| \pi(A x) h \|_{L^2} \\
  & =    \| C_{\Lambda,g} \|
         \cdot \| C_{\Lambda,\gamma} \|
         \cdot \| f \|_{L^2}
         \cdot \| h \|_{L^2} ,
\end{align*}
which will justify the application of the dominated convergence theorem in the following
calculation.
Indeed, $G$ is $\Z^{2d}$-periodic and the $k$-th Fourier coefficient of $G$ (for $k \in \Z^{2d}$)
is given by
\begin{align*}
  c_k
  &= \int_{Q}
       G(x) e^{-2\pi i k x}
     \, dx
   = \sum_{n \in \Z^{2d}}
       \int_{Q}
         \big\<f, \pi(A(n-x))g\big\>
         \big\<\pi(A(n-x))\gamma,h\big\>
         e^{2\pi ik(n-x)}
       \,d x \\
  &= \int_{\R^{2d}}
       \big\<f,\pi(Ax)g\big\>
       \big\<\pi(Ax)\gamma,h\big\>
       e^{2\pi ikx}
     \,dx
  = \frac{1}{|\det A|}
    \int_{\R^{2d}}
      V_g f(y)
      \ol{V_\gamma h(y)}
      \cdot e^{2\pi i A^{-\topp}k\cdot y}
    \, d y \\
  &= d(\La)
     \int_{\R^{2d}}
       V_{\pi(z_k)g}[\pi(z_k)f](y)
       \ol{V_\gamma h(y)}
     \, d y
   = d(\La) \cdot \<\pi(z_k)f,h\> \<\gamma,\pi(z_k)g\>,
\end{align*}
where $Q := [0,1]^{2d}$, $z_k := -JA^{-\topp}k\in\La^\iso$,
and for $f_1, g_1 \in L^2(\R^d)$, $V_{g_1} f_1 (z) = \langle f_1, \pi(z) g_1 \rangle$
for $z \in \R^{2d}$ is the \emph{short-time Fourier transform} of $f_1$ with respect to $g_1$.
Here, we used the orthogonality relation for the short-time Fourier transform
(see \mbox{\cite[Theorem~3.2.1]{GroechenigTFFoundations}})
and the identity $V_{\pi(z)g}[\pi(z)f] = e^{2\pi i\<Jz,\cdot\>}\cdot V_gf$
(\cite[Lemma~1.4.4(b)]{k}). 
Now, as also $f,g\in\calB_\La = \calB_{\La^\iso}$,
we see that $(c_k)_{k \in \Z^{2d}} \in \ell^1(\Z^{2d})$.
Since $G$ is continuous and $\Z^{2d}$-periodic,
this implies that the Fourier series of $G$ converges uniformly and coincides pointwise with $G$.
Hence,
\[
  G(x) = \sum_{k\in\Z^{2d}}
           c_k \, e^{2\pi ikx}
  \quad \text{for all} \;\; x\in\R^{2d} , 
\]
and setting $x = 0$ yields the claim.
\end{proof}


\section{Certain subspaces of modulation spaces invariant under the frame operator}
\label{s:Invariance}

The $L^2$-Sobolev-space $H^1 (\R^d) = W^{1,2}(\R^d)$ is the space of all functions
$f \in L^2(\R^d)$ whose distributional derivatives
$\partial_jf := \frac{\partial f}{\partial x_j}$, $j \in \{ 1,\ldots,d \}$,
all belong to $L^2(\R^d)$. 
We will frequently use the well-known characterization
$H^1 (\R^d) = \bigl\{ f \in L^2(\R^d) :  (1+| \cdot |) \widehat{f} (\cdot) \in L^2 \bigr\}$
of $H^1(\R^d)$ in terms of the Fourier transform.
With the weight function $w : \R^d \to \R$, $x \mapsto 1 + |x|$,
we define the weighted $L^2$-space
$L_w^2 (\R^d) := \bigl\{ f : \R^d \to \CC \colon w (\cdot) f (\cdot) \in L^2 \bigr\}$
which is equipped with the norm $\| f \|_{L_w^2} := \| w \, f \|_{L^2}$.
It is then clear that $L_w^2 (\R^d) = \mathcal{F} [ H^1 (\R^d) ] = \mathcal{F}^{-1} [ H^1 (\R^d) ]$. 
Finally, we define $\HH^1(\R^d) = H^1 (\R^d) \cap L_w^2 (\R^d)$ which is the space
of all functions $f \in H^1(\R^d)$ whose Fourier transform $\wh f$ also belongs to $H^1(\R^d)$. 
Equivalently, $\HH^1(\R^d)$ is the space of all functions $g \in L^2(\R^d)$
with finite uncertainty product \eqref{eqn:FUP}. 

It is worth to note that each of the spaces $H^1 (\R^d)$, $L_w^2 (\R^d)$,
and $\HH^1(\R^d)$ can be expressed as a modulation space 
\(
  M_m^2(\R^d)
  = \{
      f \in L^2(\R^d)
      :
      \int_{\R^{2d}}
        | \<f,\pi(z)\varphi \> |^2 \, |m(z)|^2
      \, dz
      < \infty
    \}
\)
for some weight function $m: \R^{2d} \rightarrow \CC$,
where $\varphi \in \mathcal{S} (\R^d) \backslash \{ 0 \}$ is any fixed function%
\footnote{The definition of $M_m^2$ is known to be independent of the choice of $\varphi$;
see e.g., \cite[Proposition~11.3.2]{GroechenigTFFoundations}.}, for instance a Gaussian. 
Indeed, we have
\[
  H^1(\R^d) = M_{m_1}^2(\R^d),\quad
  L_w^2(\R^d) = M_{m_2}^2(\R^d),\quad
  \text{and}\quad
  \HH^1(\R^d) = H^1 (\R^d) \cap L_w^2(\R^d) = M_{m_3}^2(\R^d),
\]
with $m_1(x,\omega) = 1+|\omega|$, $m_2(x,\omega) = 1+|x|$,
and $m_3(x,\omega) = \sqrt{1 + |x|^2 + |\omega|^2}$, respectively;
see \cite[Proposition 11.3.1]{GroechenigTFFoundations} and \cite[Corollary~2.3]{HeilTinaztepe}. 



Our main goal in this paper is to prove for each of these spaces that if the window function $g$
of a Gabor frame $(g,\Lambda)$ belongs to the space, then so does the canonical dual window.
In this section, we will mostly concentrate on the space $H^1(\R^d)$,
since this will imply the desired result for the other spaces as well.

The corresponding result for the Feichtinger algebra $S_0(\R^d)$ was proved
in \cite{GroechenigLeinert} by showing the much stronger statement that the frame operator
maps $S_0(\R^d)$ boundedly into itself and is in fact boundedly invertible on $S_0(\R^d)$.
However, the methods used in \cite{GroechenigLeinert} cannot be directly transferred
to the case of a window function in $\HH^1(\R^d)$ (or $H^1(\R^d)$),
since the proof in \cite{GroechenigLeinert} leverages 
two particular properties of the Feichtinger algebra which are not shared by $\HH^1(\R^d)$:
\begin{enumerate}
	\item[(a)] Every function from $S_0(\R^d)$ is a Bessel vector with respect to any given lattice;
	\item[(b)] The series in Janssen's representation of the frame operator converges strongly
             (even absolutely in operator norm) to the frame operator
             when the window function belongs to $S_0(\R^d)$.
\end{enumerate}
Indeed, it is well-known that $g\in L^2(\R)$ is a Bessel vector with respect to $\Z\times\Z$
if and only if the Zak transform of $g$ is essentially bounded
(cf.~\cite[Theorem~3.1]{BenedettoDifferentiationAndBLT}), but
\mbox{\cite[Example~3.4]{BenedettoDifferentiationAndBLT}} provides an example
of a function $g\in\HH^1(\R)$ whose Zak transform is not essentially bounded;
this indicates that (a) does not hold for $\HH^1(\R^d)$ instead of $S_0(\R^d)$.
Concerning the statement (b) for $\HH^1(\R^d)$, it is easy to see
that if Janssen's representation converges strongly (with respect to some enumeration of $\Z^2$)
to the frame operator of $(g,\Lambda)$, then the frame operator must be bounded on $L^2(\R)$
and thus the associated window function $g$ is necessarily a Bessel vector.
Therefore, the example above again serves as a counterexample: namely,
the statement (b) fails for such a non-Bessel window functions $g\in\HH^1(\R)$.
Even more, we show in the Appendix that there exist Bessel vectors $g\in\HH^1(\R)$
for which Janssen's representation neither converges unconditionally
in the strong sense nor conditionally in the operator norm.
%
%
%
We mention that in the case of the Wiener amalgam space $W(L^{\infty},\ell_v^1)$
with an admissible weight $v$, the convergence issue was circumvented
by employing Walnut's representation instead of Janssen's
to prove the result for $W(L^{\infty},\ell_v^1)$ in \cite{ko}. 

Fortunately, it turns out that establishing the corresponding result
for $V = H^1(\R^d)$, $L_w^2 (\R^d)$, and $\HH^1(\R^d)$ only requires the invertibility
of the frame operator on a particular subspace of $V$.
Precisely, given a lattice $\Lambda \subset \R^{2d}$, we define
\[
  H^1_\La(\R^d) := H^1(\R^d)\cap\calB_\La,
  \quad
  \HH_\La^1(\R^d) := \HH^1(\R^d)\cap\calB_\La ,
  \quad \text{and} \quad
  \WeightedBessel := L_w^2(\R^d) \cap \calB_{\Lambda} .
\]
We equip the first two of these spaces with the norms
\[
  \|f\|_{H^1_\La} := \|\nabla f\|_{L^2} + \|C_{\La,f}\|_{L^2 \to \ell^2}
  \qquad\text{and}\qquad
  \|f\|_{\HH^1_\La} := \|\nabla f\|_{L^2} + \|\nabla\wh f\|_{L^2} + \|C_{\La,f}\|_{L^2 \to \ell^2},
\]
respectively, where
\[
  \|\nabla f\|_{L^2}
  := \sum_{j=1}^d
       \|\del_jf\|_{L^2}
\]
and $C_{\Lambda, f}$ is the analysis operator defined in \eqref{eq:CoefficientOperator}.
Finally, we equip the space $\WeightedBessel$ with the norm
\[
  \NormWeightedBessel{f} := \| f \|_{L_w^2} + \| C_{\Lambda,f} \|_{L^2, \ell^2} ,
  \qquad \text{where} \qquad
  \| f \|_{L_w^2} := \| w \cdot f \|_{L^2} .
\]
We start by showing that these spaces are Banach spaces.

\begin{lem}\label{lem:BesselH1IsBanach}
  For a lattice $\Lambda\subset\R^{2d}$, the spaces $H^1_\La(\R^d)$, $\WeightedBessel$,
  and $\HH^1_\La(\R^d)$ are Banach spaces which are continuously embedded in $L^2(\R^d)$.
\end{lem}

\begin{proof}
We naturally equip the space $\calB_\La\subset L^2(\R^d)$ with the norm
$\|f\|_{\calB_\La} := \|C_{\La,f}\|_{L^2\to\ell^2}$.
Then $(\calB_\La,\|\cdot\|_{\calB_\La})$ is a Banach space by \cite[Proposition~3.1]{HanLarson}.
Moreover, for $f\in\calB_\La$,
\begin{equation}
  \|f \|_{L^2}
  = \big\| C_{\La, f}^* \, \delta_{0,0} \big\|_{L^2}
  \leq \|C_{\La,f}^*\|_{\ell^2\to L^2}
  = \|f\|_{\calB_\La},
  \label{eq:BesselEmbedsIntoL2}
\end{equation}
which implies that $\calB_\Lambda\hookrightarrow L^2(\R^d)$.
Hence, if $(f_n)_{n\in\N}$ is a Cauchy sequence in $H^1_\La(\R^d)$,
then it is a Cauchy sequence in both $H^1(\R^d)$
(equipped with the norm $\| f \|_{H^1} := \| f \|_{L^2} + \| \nabla f \|_{L^2}$) and in $\calB_\La$.
Therefore, there exist $f\in H^1(\R^d)$ and $g\in\calB_\La$
such that $\|f_n-f\|_{H^1}\to 0$ and $\|f_n-g\|_{\calB_\La}\to 0$ as $n\to\infty$.
But as $H^1(\R^d)\hookrightarrow L^2(\R^d)$ and $\calB_\La\hookrightarrow L^2(\R^d)$,
we have $f_n\to f$ and $f_n\to g$ also in $L^2(\R^d)$, which implies $f=g$.
Hence, $\|f_n - f\|_{H^1_\La}\to 0$ as $n\to\infty$, which proves that $H^1_\La(\R^d)$ is complete.
The proof for $\WeightedBessel$ and $\HH^1_\La(\R^d)$ is similar.
\end{proof}

\begin{prop}\label{p:bounded}
  Let $\Lambda\subset\R^{2d}$ be a lattice.
  If $g,h\in H^1_\La(\R^d)$, then $S_{\La,g,h}$ maps $H^1_\La(\R^d)$ boundedly into itself
  with operator norm not exceeding $\|g\|_{H^1_\La}\|h\|_{H^1_\La}$.
  For $f\in H^1_\La(\R^d)$ and $j \in \{1,\ldots,d\}$ we have
  \begin{align}\label{e:DSf}
    \del_j(S_{\La,g,h}f)
    &= S_{\La,g,h}(\del_jf) + d(\La)\cdot C_{\La^\iso,f}^* \, d_{j,\La^\iso,g,h},
  \end{align}
  where $d_{j,\Lambda^{\iso},g,h} \in \ell^2(\Lambda^{\iso})$ is defined by
  \begin{align}\label{e:de}
    (d_{j,\La^\iso,g,h})_{\mu}
    := \big\<\del_jh,\pi(\mu)g\big\>
       + \big\<h,\pi(\mu)(\del_jg)\big\>,\quad \mu\in\La^\iso.
  \end{align}
\end{prop}

\begin{proof}
Let $f\in H^1_\La(\R^d)$ and set $u := S_{\La,g,h}f$.
First of all, we have $u \in \calB_\La$.
Indeed, a direct computation shows that $S_{\Lambda,g,h}$ commutes with $\pi (\la)$
for all $\la\in\La$, and that $S_{\Lambda, g, h}^\ast = S_{\Lambda, h, g}$,
which shows for $v \in L^2(\R^d)$ that
\[
  (C_{\Lambda,u} v)_{\lambda}
  = \< v, \pi(\la) u\>
  = \< v,\pi(\la) S_{\La,g,h} f \>
  = \< S_{\La,h,g} v, \pi(\la) f \>
  = (C_{\Lambda,f} \circ S_{\Lambda,h,g} \, v)_{\lambda},
\]
and therefore
\begin{equation}\label{e:alledrei}
  \| C_{\La,u} \|
  \leq \|S_{\La,h,g}\| \cdot \|C_{\La,f}\|
  \leq \| C_{\La,g} \| \cdot \| C_{\La,h} \| \cdot \| C_{\La,f} \| < \infty ,
\end{equation}
since $S_{\La,h,g} = C_{\La,g}^*C_{\La,h}$.

We now show that $u \in H^1(\R^d)$.
To this end, note for $v \in H^1(\R^d)$, $a, b \in \R^d$, and $j \in \{1,\ldots,d\}$ that
\[
  \del_j (M_b v)
  = 2\pi i \cdot b_j \cdot M_{b}v + M_{b}(\del_jv)
  \qquad\text{and}\qquad
  \del_j(T_av) = T_{a}(\del_jv)
\]
and therefore
\[
  \del_j(\pi(z)v)
  = 2\pi i \cdot z_{d+j}\cdot\pi(z)v + \pi(z)(\del_jv).
\]
Hence, setting $c_{\la,j} := 2\pi i \cdot \la_{d+j}\cdot\<f,\pi(\la)g \>$ for $\la = (a,b)\in \La$,
we see that
\begin{align}
\begin{split}\label{e:cmn}
  c_{\la,j}
  &= \<\del_jf,\pi(\la)g\> + \<f,\pi(\la)(\del_jg)\>.
\end{split}
\end{align}
In particular, $(c_{\la,j})_{\la\in\La}\in\ell^2(\La)$ for each $j\in\{1,\ldots,d\}$,
because $f,g\in\calB_\La$ and $\partial_j f , \partial_j g \in L^2$.

In order to show that $\del_ju$ exists and is in $L^2(\R^d)$,
let $\phi\in C_c^\infty(\R^d)$ be a test function.
Note that $C_c^\infty(\R^d)\subset \calB_\La$.
Therefore, we obtain
\begin{align*}
  -\big\<u,\del_j\phi\big\>
  &= - \!\!
       \sum_{\la\in\La}
         \<f,\pi(\la)g\big\>
         \big\<\pi(\la)h,\del_j\phi\big\>
  = \sum_{\la\in\La}
      \<f,\pi(\la)g\big\>
      \big\<
        2\pi i\la_{d+j}\cdot\pi(\la)h
        + \pi(\la)(\del_jh)
        ,
        \phi
      \big\>\\
  &= \sum_{\la\in\La}
       c_{\la,j}
       \cdot \big\<\pi(\la)h,\phi\big\>
     + \sum_{\la\in\La}
         \big\<f,\pi(\la)g\big\>
         \big\<\pi(\la)(\del_jh),\phi\big\> \\
  & \!\!\overset{\eqref{e:cmn}}{=}
    \<S_{\La,g,h}(\del_jf),\phi\>
    + \sum_{\la\in\La}
        \<f,\pi(\la)(\del_jg)\>
        \<\pi(\la)h,\phi\>
    + \sum_{\la\in\La}
        \big\<f,\pi(\la)g\big\>
        \big\<\pi(\la)(\del_jh), \phi\big\> \\
  & \!\!\overset{\eqref{e:fi}}{=}
    \<S_{\La,g,h}(\del_jf),\phi\>
    + d(\La) \sum_{\mu\in\La^\iso}
               \Big[
                 \big\<h,\pi(\mu)(\del_jg)\big\>
                 + \big\<\del_jh,\pi(\mu)g\big\>
               \Big]
               \big\<\pi(\mu)f,\phi\big\> \\[-0.1cm]
  &= \left\<
       S_{\La,g,h} (\del_jf)
       + d(\La) \sum_{\mu\in\La^\iso}
                  \Big[
                    \big\<h,\pi(\mu)(\del_jg)\big\>
                    + \big\<\del_jh,\pi(\mu)g\big\>
                  \Big]
                  \pi(\mu)f
       \,,\,\phi
     \right\> \\
  &= \left\<
       S_{\La,g,h}(\del_jf)
       + d(\La)\cdot C_{\La^\iso,f}^* \, d_j
       \,,\phi
     \right\>,
\end{align*}
with $d_j = d_{j,\La^\iso,g,h}$ as in \eqref{e:de}.
Note that $d_j\in \ell^2(\La^\iso)$ because $g,h\in\calB_{\Lambda} = \calB_{\Lambda^\circ}$
and ${\partial_j h , \partial_j g \in L^2}$.
Since $j\in\{1,\ldots,d\}$ is chosen arbitrarily, this proves that $u\in H^1(\R^d)$ with
\[
  \del_ju
  = S_{\La,g,h}(\del_jf)
    + d(\La)\cdot C_{\La^\iso,f}^* \, d_j
  \,\in\, L^2(\R^d)
\]
for $j \in \{1,\ldots,d\}$, which is \eqref{e:DSf}.
Next, recalling Equation \eqref{e:Cnorms} we get
\[
  \|d_j\|_{\ell^2}
  \leq \|C_{\La^\iso,h}\| \cdot \|\del_jg\|_{L^2}
       + \|C_{\La^\iso,g}\| \cdot \|\del_jh\|_{L^2}
  = \Vol(\La)^{1/2}
    \big(
      \|C_{\La,h}\| \cdot \|\del_jg\|_{L^2}
      + \|C_{\La,g}\| \cdot \|\del_jh\|_{L^2}
    \big),
\]
and $\|C_{\La^\iso,f}^*\| = \Vol(\La)^{1/2}\|C_{\La,f}\|$.
Therefore,
\begin{align*}
  \|\del_ju\|_{L^2}
  \leq \|S_{\La,g,h}\| \cdot \|\del_jf\|_{L^2}
       + \big(
           \|C_{\La,h}\| \cdot \|\del_jg\|_{L^2}
           + \|C_{\La,g}\| \cdot \|\del_jh\|_{L^2}
         \big)
         \, \|C_{\La,f}\|.
\end{align*}
Hence, with \eqref{e:alledrei}, we see
\vspace*{-0.3cm}
\begin{align*}
  \|S_{\La,g,h}f\|_{H^1_\La}
  &= \|\nabla u\|_{L^2}
     + \|C_{\La,u}\|
   \leq \sum_{j=1}^d \|\del_ju\|_{L^2}
        + \|C_{\La,g}\| \cdot \|C_{\La,h}\| \cdot \|C_{\La,f}\| \\
  &\leq \|S_{\La,g,h}\| \!\cdot\! \|\nabla f\|_{L^2}
        + \big(
            \|C_{\La,h}\| \!\cdot\! \|\nabla g\|_{L^2}
            \!+\! \|C_{\La,g}\| \!\cdot\! \|\nabla h\|_{L^2}
            \!+\! \|C_{\La,g}\| \!\cdot\! \|C_{\La,h}\|
          \big)
          \|C_{\La,f}\| \\
  &\leq \|C_{\La,g}\| \cdot \|C_{\La,h}\| \cdot \|\nabla f\|_{L^2}
        + \big(
            \|\nabla g\|_{L^2} + \|C_{\La,g}\|
          \big)
          \big(
            \|\nabla h\|_{L^2}
            + \|C_{\La,h}\|
          \big)
          \, \|C_{\La,f}\| \\
  &\leq \|g\|_{H^1_\La}
        \|h\|_{H^1_\La}
        \cdot \|f\|_{H^1_\La},
\end{align*}
and the proposition is proved.
\end{proof}

\section{Spectrum and dual windows}
\label{s:spectra}

Let $X$ be a Banach space.
As usual, we denote the set of bounded linear operators from $X$ into itself by $\calB(X)$.
The {\em resolvent set} $\rho(T)$ of an operator $T\in\calB(X)$ is the set of all $z\in\CC$
for which $T-z := T-z I : X \to X$ is bijective.
Note that $\rho(T)$ is always open in $\CC$.
The {\em spectrum} of $T$ is the complement $\sigma(T) := \CC\backslash\rho(T)$.
The {\em approximate point spectrum} $\sap(T)$ is a subset of $\sigma(T)$
and is defined as the set of points $z\in\CC$ for which there exists a sequence
$(f_n)_{n\in\N}\subset X$ such that $\|f_n\|=1$ for all $n\in\N$
and $\|(T-z)f_n\|\to 0$ as $n\to\infty$.
By \mbox{\cite[Proposition~VII.6.7]{ConwayFA}} we have
\begin{equation}\label{e:partial_sap}
  \partial\sigma(T)\,\subset\,\sap(T).
\end{equation}

\begin{lem}\label{l:spectra_eq}
  Let $(\calH, \| \cdot \|)$ be a Hilbert space, let $S\in\calB(\calH)$ be self-adjoint,
  and let $X\subset\calH$ be a dense linear subspace satisfying $S(X) \subset X$.
  If $\|\cdot\|_X$ is a norm on $X$ such that $(X,\| \cdot \|_X)$ is complete
  and satisfies $X\hookrightarrow\calH$, then $A := S|_X\in\calB(X)$.
  If, in addition, $\sap(A)\subset\sigma(S)$, then $\sigma(A) = \sigma(S)$.
\end{lem}

\begin{proof}
The fact that $A \in \calB(X)$ easily follows from the closed graph theorem.
Next, since $X \hookrightarrow \calH$, there exists $C > 0$
with $\| f \| \leq C \, \| f \|_X$ for all $f \in X$.
Assume now that additionally $\sap(A) \subset \sigma(S)$ holds.
Note that $\sigma(S) \subset \R$, since $S$ is self-adjoint.
Since $\sigma(A)\subset\CC$ is compact, the value ${r := \max_{w\in\sigma(A)}|\!\Im w|}$ exists.
Choose $z\in\sigma(A)$ such that $|\!\Im z| = r$.
Clearly, $z$ cannot belong to the interior of $\sigma(A)$, and hence $z \in \partial \sigma(A)$.
In view of Equation~\eqref{e:partial_sap}, this implies $z \in \sap(A) \subset \sigma(S) \subset \R$,
hence $r = 0$ and thus $\sigma(A) \subset \R$.
Therefore, $\sigma(A)$ has empty interior in $\CC$, meaning $\sigma (A) = \partial \sigma (A)$.
Thanks to Equation~\eqref{e:partial_sap}, this means $\sigma(A) \subset \sap(A)$,
and hence $\sigma(A) \subset \sigma(S)$, since by assumption $\sap(A) \subset \sigma(S)$.

For the converse inclusion it suffices to show that $\rho(A)\cap\R\subset\rho(S)$.
To see that this holds, let $z\in\rho(A)\cap\R$ and denote by $E$ the spectral measure
of the self-adjoint operator $S$.
Since $\R \cap \rho(A) \subset \R$ is open, there are
$a, b \in \R$ and $\delta_0 > 0$ such that $z \in (a,b)$
and $[a-\delta_0, b+\delta_0] \subset \rho(A)$.
By Stone's formula (see, e.g., \cite[Thm.\ VII.13]{rs}),
the spectral projection of $S$ with respect to $(a,b]$ can be expressed as
\[
  E((a,b])f
  = \lim_{\delta\downto 0}\,
      \lim_{\veps\downto 0}
        \frac{1}{2\pi i}
        \int_{a+\delta}^{b+\delta}
          \big[
            (S-t-i\veps)^{-1}f
            -(S-t+i\veps)^{-1}f
          \big]
        \,dt,
  \qquad f\in\calH,
\]
where all limits are taken with respect to the norm of $\calH$.

Note for $w \in \CC \setminus \R$ that $w \in \rho(S) \subset \rho(A)$.
Furthermore, $A - w = (S - w)|_X$, which easily implies $(S - w)^{-1}|_X = (A - w)^{-1}$.
Hence, for $f \in X$,
\begin{align*}
  \|E((a,b])f\|
  &\leq \lim_{\delta\downto 0}\,
          \lim_{\veps\downto 0}
            \frac{1}{2\pi}
            \int_{a+\delta}^{b+\delta}
              \big\|(S- t-i\veps)^{-1}f-(S- t+i\veps)^{-1}f\big\|
            \,d t\\
  &\leq C \cdot \lim_{\delta\downto 0}\,
                  \lim_{\veps\downto 0}
                    \frac{1}{2\pi}
                    \int_{a+\delta}^{b+\delta}
                      \big\|(A- t-i\veps)^{-1}f-(A- t+i\veps)^{-1}f\big\|_{X}
                    \,d t \\
  &= C \cdot \lim_{\delta\downto 0}
               \frac{1}{2\pi}
               \int_{a+\delta}^{b+\delta}
                 \lim_{\veps\downto 0}
                   \big\|(A- t-i\veps)^{-1}f-(A- t+i\veps)^{-1}f\big\|_{X}
               \,d t\\
  &=0,
\end{align*}
since the map $\rho(A)\to X$, $z\mapsto (A-z)^{-1}f$ is analytic
and thus uniformly continuous on compact sets.
This implies $E((a,b])f = 0$ for all $f\in X$ and therefore $E((a,b]) = 0$
as $X$ is dense in $\calH$.
But this means that $(a,b) \subset \rho(S)$ (see \cite[Prop.\ on p.\ 236]{rs})
and thus $z \in \rho(S)$.
\end{proof}


For proving the invertibility of $S_{\Lambda,g}$ on $H_{\Lambda}^1, L_{w,\Lambda}^2$,
and $\HH^1_{\Lambda}$, we first focus on the space $H^1_\La(\R^d)$. 
Note that if $g \in H^1_\La(\R^d)$, then $S_{\La,g}$ maps $H^1_\La(\R^d)$ boundedly into itself
by Proposition~\ref{p:bounded}.
For $g \in H^1_\La(\R^d)$, we will denote the restriction of $S_{\La,g}$ to $H^1_\La(\R^d)$
by $A_{\La,g}$; that is, $A_{\La,g} := S_{\La,g} |_{H^1_\La(\R^d)} \in\calB(H^1_\La(\R^d))$.

\begin{thm}\label{t:spectra}
  Let $\Lambda\subset\Z^{2d}$ be a lattice and let $g\in H^1_\La(\R^d)$.
  Then
  \[
    \sigma(A_{\La,g}) = \sigma(S_{\La,g}).
  \]
\end{thm}

\begin{proof}
For brevity, we set $A := A_{\La,g}$ and $S := S_{\La,g}$.
Due to Lemma~\ref{l:spectra_eq}, we only have to prove that $\sap(A)\subset\sigma(S)$.
For this, let $z\in\sap(A)$.
Then there exists a sequence $(f_n)_{n\in\N} \subset H^1_\La(\R^d)$ such that
$\|f_n\|_{H^1_\La} = 1$ for all $n \in \N$ and $\|(A - z)f_n\|_{H^1_\La}\to 0$ as $n \to \infty$.
The latter means that, for each $j \in \{1, \ldots, d\}$,
\begin{equation}\label{e:two}
  \big\| \del_j(S f_n) -  z \cdot (\del_j f_n) \big\|_{L^2}\to 0
  \qquad\text{and}\qquad
  \big\| C_{\La,(S -  z)f_n} \big\|\to 0.
\end{equation}
Suppose towards a contradiction that $z \notin \sigma(S)$.
Since $S$ is self-adjoint, this implies $\overline{z} \notin \sigma(S)$.
Furthermore, because $S$ is self-adjoint and commutes with $\pi(\la)$ for all $\la \in \La$,
we see for $f \in \calB_\La$ that $C_{\La,(S - z)f} = C_{\La,f} \circ (S - \ol z)$
and hence $C_{\La,f_n} = C_{\La,(S-z)f_n} \circ (S - \ol z)^{-1}$,
which implies that $\|C_{\La,f_n}\|\to 0$.
Hence, also $\|C_{\La^\iso,f_n}\|\to 0$ as $n\to\infty$ (see Equation~\eqref{e:Cnorms}).
Now, by Equation~\eqref{e:DSf}, we have
\[
  \del_j(Sf_n) - z \cdot (\del_j f_n)
  = (S- z)(\del_jf_n) + C_{\La^\iso,f_n}^*d_j
\]
with some $d_j\in\ell^2(\La^\iso)$ which is \emph{independent of} $n$.
Hence, the first limit in \eqref{e:two} combined with $\|C_{\La^\iso,f_n}\|\to 0$
implies that $\|(S -  z)(\del_jf_n)\|_{L^2}\to 0$ and thus $\|\del_jf_n\|_{L^2}\to 0$
as $n\to\infty$ for all $j \in \{1,\ldots,d\}$, since $z \notin \sigma(S)$.
Hence, $\|f_n\|_{H^1_\La} = \sum_{j=1}^d \|\del_j f_n\|_{L^2} + \|C_{\La,f_n}\|\to 0$
as $n\to\infty$, in contradiction to $\|f_n\|_{H^1_\La} = 1$ for all $n \in \N$.
This proves that, indeed, $\sap(A)\subset\sigma(S)$.
\end{proof}

We now show analogous properties to Proposition~\ref{p:bounded} and Theorem~\ref{t:spectra}
for $\WeightedBessel$. 

\begin{cor}\label{cor:WeightedSpaceSpectrum} 
  Let $\Lambda\subset\Z^{2d}$ be a lattice.
  If $g, h \in \WeightedBessel$, then $S_{\Lambda,g,h}$ maps $\WeightedBessel$
  boundedly into itself.
  If $g = h$ and if $A^w_{\La,g} := S_{\La,g} |_{\WeightedBessel} \in \calB(\WeightedBessel)$
  denotes the restriction of $S_{\La,g}$ to $\WeightedBessel$, then
  \[
    \sigma(A^w_{\La,g}) = \sigma(S_{\La,g}).
  \]
\end{cor}

\begin{proof}
  We equip the space $\calB_{\Lambda} \subset L^2(\R^d)$ with the norm
  ${\| f \|_{\calB_\Lambda} := \| C_{\Lambda,f} \|_{L^2 \to \ell^2}}$,
  where we recall from Equation~\eqref{eq:BesselEmbedsIntoL2} that
  $\| f \|_{L^2} \leq \| f \|_{\calB_\Lambda}$.
  Equation~\eqref{e:FTpi} shows that the Fourier transform
  is an isometric isomorphism from $\calB_{\La}$ to $\calB_{\widehat{\La}}$,
  where $\widehat{\Lambda} := J \Lambda$.
  Furthermore, it is well-known (see for instance \mbox{\cite[Section~9.3]{FollandRealAnalysis}})
  that the Fourier transform ${\Fourier : L^2 \to L^2}$
  restricts to an isomorphism of Banach spaces ${\Fourier : L_{w}^2(\R^d) \to H^1(\R^d)}$,
  where $H^1$ is equipped with the norm $\| f \|_{H^1} := \| f \|_{L^2} + \| \nabla f \|_{L^2}$.
  Taken together, we thus see that the Fourier transform restricts to an isomorphism
  $\Fourier : \WeightedBessel \to H_{\widehat{\Lambda}}^1(\R^d)$; here, we implicitly used that
  ${\| f \|_{H_{\widehat{\Lambda}}^1} \asymp \| f \|_{H^1} + \| f \|_{\calB_{\widehat{\Lambda}}}}$,
  which follows from $\| \cdot \|_{L^2} \leq \| \cdot \|_{\calB_{\widehat{\Lambda}}}$.

  Plancherel's theorem, in combination with Equation~\eqref{e:FTpi}
  shows for $f \in L^2(\R^d)$ that
  \begin{align*}
    \Fourier \bigl[S_{\Lambda,g,h} f\bigr]
    \!=\! \sum_{\la\in\La}
            \Big\<\wh f,\wh{\pi(\la)g}\Big\>
            \wh{\pi(\la)h}
    =\! \sum_{\la\in\La}
          \<\wh f,\pi(J\la)\wh g \,\>
          \pi(J\la)\wh h
    =\! \sum_{\la\in\wh\La}
          \<\wh f,\pi(\la)\wh g \,\>
          \pi(\la)\wh h
    = S_{\wh\La,\wh g,\wh h} \wh f.
  \end{align*}
  Since
  \(
    A_{\widehat{\Lambda},\widehat{g},\widehat{h}}
    = S_{\widehat{\Lambda},\widehat{g},\widehat{h}}|_{H_{\widehat{\Lambda}}^1} :
    H_{\widehat{\Lambda}}^1(\R^d) \to H_{\widehat{\Lambda}}^1(\R^d)
  \)
  is well-defined and bounded by Proposition~\ref{p:bounded}, the preceding calculation
  combined with the considerations from the previous paragraph shows that
  $A_{\Lambda,g,h}^w = S_{\Lambda,g,h}|_{\WeightedBessel} : \WeightedBessel \to \WeightedBessel$
  is well-defined and bounded, with
  \[
    A_{\Lambda,g,h}^w
    = \Fourier^{-1} \circ A_{\widehat{\Lambda},\widehat{g},\widehat{h}} \circ \Fourier .
  \]
  Finally, if $g = h$, we see
  \(
    \sigma(A_{\Lambda,g,g}^w)
    = \sigma(A_{\widehat{\Lambda},\widehat{g},\widehat{g}})
    = \sigma(S_{\widehat{\Lambda},\widehat{g},\widehat{g}})
    = \sigma(S_{\Lambda,g,g}) ,
  \)
  where the second step is due to Theorem~\ref{t:spectra}, and the final step used the identity
  $S_{\Lambda,g,h} = \Fourier^{-1} \circ S_{\widehat{\Lambda},\widehat{g},\widehat{h}} \circ \Fourier$
  from above.
\end{proof}

Finally, we establish the corresponding properties for
$\HH^1_\La(\R^d) = H^1_\La(\R^d) \cap \WeightedBessel$.  

\begin{cor}\label{c:bounded_HH}
  Let $\Lambda\subset\Z^{2d}$ be a lattice.
  If $g,h\in\HH^1_\La(\R^d)$, then $S_{\La,g,h}$ maps $\HH^1_\La(\R^d)$ boundedly into itself.
  If $g=h$ and $\mathbb A_{\La,g} := S_{\La,g} |_{\HH^1_\La(\R^d)} \in\calB(\HH^1_\La(\R^d))$
  denotes the restriction of $S_{\La,g}$ to $\HH^1_\La(\R^d)$, then
  \begin{equation}
    \sigma(\mathbb A_{\La,g}) = \sigma(S_{\La,g}).
    \label{eq:FatHSpectrum}
  \end{equation}
\end{cor}
\begin{proof}
  From the definition of $\HH_\Lambda^1$ and the proof of Corollary~\ref{cor:WeightedSpaceSpectrum}
  it is easy to see that $\HH_\Lambda^1 = H_\Lambda^1 \cap \WeightedBessel$,
  and $\| \cdot \|_{\HH_\Lambda^1} \asymp \| \cdot \|_{H_\Lambda^1} + \NormWeightedBessel{\cdot}$.
  Therefore, Proposition~\ref{p:bounded} and Corollary~\ref{cor:WeightedSpaceSpectrum} imply
  that $S_{\Lambda,g,h}$ maps $\HH_\Lambda^1(\R^d)$ boundedly into itself.

  Lemma~\ref{l:spectra_eq} shows that to prove \eqref{eq:FatHSpectrum},
  it suffices to show $\sap(\mathbb{A}_{\Lambda,g}) \subset \sigma(S_{\Lambda,g})$.
  Thus, let ${z \in \sap(\mathbb{A}_{\Lambda,g})}$.
  Then there exists $(f_n)_{n \in \N} \subset \HH^1_\La(\R^d)$ with $\|f_n\|_{\HH_{\Lambda}^1}=1$
  for all $n \in \N$ and $\|(\mathbb A_{\La,g}-z)f_n\|_{\HH^1_\La}\to 0$ as $n\to\infty$.
  Thus, $\|(A_{\La,g}-z)f_n\|_{H^1_\La}\to 0$
  and ${\NormWeightedBessel{(A_{\La,g}^w - z)f_n} \to 0}$ as $n\to\infty$.
  Furthermore, there is a subsequence $(n_k)_{k \in \N}$ such that
  $\lim_{k\to\infty} \| f_{n_k} \|_{H_\Lambda^1} > 0$
  or $\lim_{k\to\infty} \NormWeightedBessel{f_{n_k}} > 0$.
  Hence, $z \in \sigma(A_{\La,g})$ or $z \in \sigma(A_{\La,g}^w)$.
  But Theorem~\ref{t:spectra} and Corollary~\ref{cor:WeightedSpaceSpectrum} show
  $\sigma(A_{\La,g}) = \sigma(A_{\La,g}^w) = \sigma(S_{\Lambda,g})$.
  We have thus shown $\sap(\mathbb{A}_{\Lambda,g}) \subset \sigma(S_{\Lambda,g})$,
  so that Lemma~\ref{l:spectra_eq} shows $\sigma(\mathbb{A}_{\Lambda,g}) = \sigma(S_{\Lambda,g})$.
\end{proof}

The next proposition shows that any operator obtained from
$S_{\Lambda,g}$ through the holomorphic spectral calculus
(see \cite[Sections~10.21--10.29]{RudinFunctionalAnalysis} for a definition)
maps each of the spaces $H_{\Lambda}^1(\R^d)$, $\WeightedBessel$,
and $\HH_\Lambda^1(\R^d)$ into itself.

\begin{prop}\label{t:inverse_closed}
  Let $\Lambda\subset\R^{2d}$ be a lattice,
  let $V \in \{ H_\Lambda^1(\R^d), \WeightedBessel, \HH_\Lambda^1(\R^d) \}$, and $g \in V$.
  Then for any open set $\Omega \subset \CC$ with $\sigma(S_{\Lambda,g}) \subset \Omega$,
  any analytic function $F : \Omega \to \CC$, and any $f \in V$, we have $F(S_{\La,g})f \in V$.
\end{prop}

\begin{proof}
  We only prove the claim for $V = H_{\Lambda}^1 (\R^d)$;
  the proofs for the other cases are similar, using Corollaries~\ref{cor:WeightedSpaceSpectrum}
  or \ref{c:bounded_HH} instead of Theorem~\ref{t:spectra}.
  Thus, let $g\in H^1_\La(\R^d)$ and set $S := S_{\La,g}$ and $A := A_{\La,g}$.
  Let $f\in H^1_\La(\R^d)$ and define
  \[
    h
    = - \frac{1}{2\pi i}
        \int_\Gamma
          F(z) \cdot (A-z)^{-1}f
        \,dz
    \,\in\, H^1_\La(\R^d),
  \]
  where $\Gamma \subset \Omega \setminus \sigma(S)$ is a finite set of closed rectifiable curves
  surrounding $\sigma(S) = \sigma(A)$
  (existence of such curves is shown in \cite[Theorem~13.5]{RudinRealAndComplexAnalysis}).
  Note that the integral converges in $H^1_\La(\R^d)$.
  Since $H^1_\La(\R^d)\hookrightarrow L^2(\R^d)$, it also converges
  (to the same limit) in $L^2(\R^d)$ and hence, by definition of the holomorphic spectral calculus,
  \[
    F(S)f
    = -\frac{1}{2\pi i}
       \int_\Gamma
        F(z) \cdot (S-z)^{-1}f
      \,dz
    = h\,\in\,H^1_\La(\R^d).
    \qedhere
  \]
\end{proof}

%

Our main result (Theorem~\ref{t:main}) is now an easy consequence of Proposition~\ref{t:inverse_closed}.


\begin{proof}[Proof of Theorem~\ref{t:main}]
  Using the fact that $S_{\La,g}$ commutes with $\pi(\la)$ for all $\la \in \La$,
  it is easily seen that $(S_{\La,g}^{-1} \, g,\La)$ is the canonical dual frame of $(g,\La)$
  and that $(S_{\La,g}^{-1/2} g, \Lambda)$ is a Parseval frame for $L^2(\R^d)$;
  see for instance, \cite[Theorem 12.3.2]{ChristensenBook}.
  Note that since $(g,\La)$ is a frame for $L^2(\R^d)$,
  we have $\sigma(S_{\Lambda,g}) \subset [A,B]$
  where $0 < A \leq B < \infty$ are the optimal frame bounds for $(g,\La)$.
  Thus, we obtain $S_{\La,g}^{-1} \, g \in V_\Lambda \subset V$
  and $S_{\La,g}^{-1/2} g \,\in\, V_\Lambda \subset V$
  from Proposition~\ref{t:inverse_closed} with $F(z) = z^{-1}$ and $F(z) = z^{-1/2}$
  (with any suitable branch cut; for instance, the half-axis $(-\infty,0]$),
  respectively, on $\Omega = \bigl\{ x + i y : x \in ( \frac{A}{2}, \infty) , y \in \R\bigr\}$. 
\end{proof}

Finally, we state and prove a version of \Cref{t:main} for Gabor frame \emph{sequences}.
For completeness, we briefly recall the necessary concepts.
Generally, a (countable) family $(h_i)_{i \in I}$ in a Hilbert space $\calH$
is called a \emph{frame sequence}, if $(h_i)_{i \in I}$ is a frame for the subspace
$\calH' := \overline{\linspan} \{ h_i \colon i \in I \} \subset \calH$.
In this case, the frame operator
$S : \calH \to \calH, f \mapsto \sum_{i \in I} \langle f, h_i \rangle h_i$,
is a bounded, self-adjoint operator on $\calH$, and $S|_{\calH'} : \calH' \to \calH'$
is boundedly invertible; in particular, $\ran S = \calH' \subset \calH$ is closed,
so that $S$ has a well-defined \emph{pseudo-inverse} $S^{\dagger}$, given by
\[
  S^{\dagger}
  = (S|_{\calH'})^{-1} \circ P_{\calH'} : \quad
  \calH \to \calH' ,
\]
where $P_{\calH'}$ denotes the orthogonal projection onto $\calH'$.
The \emph{canonical dual system} of $(h_i)_{i \in I}$ is then given by
$(h_i')_{i \in I} = (S^{\dagger} h_i)_{i \in I} \subset \calH'$, and it satisfies
$\sum_{i \in I} \langle f, h_i \rangle h_i ' = \sum_{i \in I} \langle f, h_i' \rangle h_i = P_{\calH'} f$
for all $f \in \calH$.

Finally, in the case where $(h_i)_{i \in I} = (g,\Lambda)$ is a Gabor family with a lattice $\Lambda$,
it is easy to see that $S \circ \pi(\lambda) = \pi(\lambda) \circ S$ and
$\pi(\lambda) \calH' \subset \calH'$ for $\lambda \in \Lambda$, which implies
$P_{\calH'} \circ \pi(\lambda) = \pi(\lambda) \circ P_{\calH'}$,
and therefore $S^\dagger \circ \pi(\lambda) = \pi(\lambda) \circ S^\dagger$
for all $\lambda \in \Lambda$.
Consequently, setting $\gamma := S^\dagger g$, we have $S^\dagger (\pi(\lambda) g) = \pi(\lambda)\gamma$,
so that the canonical dual system of a Gabor frame sequence $(g,\Lambda)$ is
the Gabor system $(\gamma,\Lambda)$, where $\gamma = S^{\dagger} g$ is called the
\emph{canonical dual window} of $(g,\Lambda)$.
Our next result shows that $\gamma$ inherits the regularity of $g$,
if one measures this regularity using one of the three spaces $H^1, L_w^2$, or $\HH^1$.

\begin{prop}\label{p:MainResultForGaborFrameSequences}
  Let $V \in \{ H^1(\R^d), L_w^2(\R^d), \HH^1(\R^d) \}$.
  Let $\Lambda \subset \R^{2d}$ be a lattice and let $g \in V$.
  If $(g,\Lambda)$ is a frame sequence, then the associated canonical dual window $\gamma$
  satisfies $\gamma \in V$.
\end{prop}


\begin{proof}
  The frame operator $S : L^2(\R^d) \to L^2(\R^d)$ associated to $(g,\Lambda)$
  is non-negative and has closed range.
  Consequently, there exist $\eps > 0$ and $R > 0$ such that $\sigma(S) \subset \{ 0 \} \cup [\eps,R]$;
  see for instance \cite[Lemma~A.2]{QuantitativeSubspaceBL}.
  Now, with the open ball $B_\delta(0) := \{ z \in \CC \colon |z| < \delta \}$, define
  \[
    \Omega
    := B_{\eps/4} (0)
       \cup \big\{
              x + i y
              \colon
              x \in (\tfrac{\eps}{2},2 R), y \in (-\tfrac{\eps}{4}, \tfrac{\eps}{4})
            \big\}
    \subset \CC ,
  \]
  noting that $\Omega \subset \CC$ is open, with $\sigma(S) \subset \Omega$.
  Furthermore, it is straightforward to see that
  \[
    \varphi : \quad
    \Omega \to \CC, \quad
    z \mapsto \begin{cases}
                0,      & \text{if } z \in B_{\eps/4} (0), \\
                z^{-1}, & \text{otherwise}
              \end{cases}
  \]
  is holomorphic.
  Since the functional calculus for self-adjoint operators
  is an extension of the holomorphic functional calculus,
  \cite[Lemma~A.6]{QuantitativeSubspaceBL} shows that $S^{\dagger} = \varphi(S)$.
  Finally, since $g \in V_\Lambda$, \Cref{t:inverse_closed} now shows that
  $\gamma = S^{\dagger} g = \varphi(S) g \in V_\Lambda \subset V$ as well.
\end{proof}

\appendix

\section{(Non)-convergence of Janssen's representation for \texorpdfstring{$\HH^1$}{ℍ¹} windows}

In this appendix we provide a counterexample showing that Janssen's representation
of the frame operator associated to a Bessel vector $g \in \HH^1$
in general does not converge \emph{unconditionally} with respect to the strong operator topology.
We furthermore show that for convergence in \emph{operator norm},
even conditional convergence fails in general.

For simplicity, we only consider the setting $d = 1$ and the lattice $\Lambda = \Z \times \Z$.
Thus, given a function $g \in \HH^1 = \HH^1(\R)$, we say that $g$ is a \emph{Bessel vector}
if the Gabor system $(T_k M_\ell \, g)_{k,\ell \in \Z} \subset L^2(\R)$
is a Bessel system.
In this case, \emph{Janssen's representation} of the frame operator
$S := S_g := S_{\Z \times \Z, g, g}$ is (formally) given by
\begin{equation}
  S
  = \sum_{\ell, k \in \Z}
      \langle g, T_k M_\ell \, g \rangle \, T_k M_\ell.
  \label{eq:JannssenRepresentation}
\end{equation}
We are interested in the question whether the series defining Janssen's representation
is unconditionally convergent in the \emph{strong operator topology \braces{SOT\,}},
as an operator on $L^2(\R)$. 
We will construct a function $g \in \HH^1$ for which this fails. 

\subsection{Properties of the Zak transform}\label{sec:ZakProperties}

The construction of the counterexample is based on several properties of the \emph{Zak transform}
that we briefly recall.
Given $f \in L^2(\R)$, its Zak transform $Z f \in L_{{\rm loc}}^2(\R^2)$ is defined as
\[
  Z f (x,\omega)
  := \sum_{k \in \Z}
       f(x - k) e^{2 \pi i k \omega} ,
\]
where the series converges in $L_{{\rm loc}}^2 (\R^2)$;
this is a consequence of the fact that
\begin{equation}
  Z : \quad
  L^2(\R) \to L^2([0,1]^2)
  \text{ is unitary} ,
  \label{eq:ZakIsUnitary}
\end{equation}
as shown in \cite[Theorem~8.2.3]{GroechenigTFFoundations}
and of the fact that the Zak transform $Z f$
of a function $f \in L^2(\R)$ is always \emph{quasi-periodic}, meaning that
\begin{equation}
  Z f(x+n, \omega) = e^{2 \pi i n \omega} Z f (x,\omega)
  \qquad \text{and} \qquad
  Z f (x, \omega + n) = Z f (x,\omega) 
  \label{eq:QuasiPeriodicity}
\end{equation}
for (almost) all $x,\omega \in \R$ and all $n \in \Z$;
see \cite[Equations~(8.4) and (8.5)]{GroechenigTFFoundations}.
Another crucial property is the interplay between the Zak transform and the time-frequency
shifts $T_k M_n$, as expressed by the following formula
(found in \cite[Equation~(8.7)]{GroechenigTFFoundations}):
\begin{equation}
  Z[T_k M_n f] (x,\omega)
  = e^{2 \pi i n x} e^{- 2 \pi i k \omega} Z f (x,\omega)
  = e_{n,-k} (x,\omega) \, Z f (x,\omega) ,
  \label{eq:ZakDiagonalizesTF}
\end{equation}
where we used the functions
\[
  e_{n,k}(x,\omega)
  := e^{2 \pi i(nx+ k\omega)}
  \quad \text{for } n,k\in\Z \text{ and } x,\omega \in \R.
\]
Note that $(e_{n,k})_{n,k\in\Z}$ is an orthonormal basis of $L^2 ([0,1]^2)$.

Finally, we note the following equivalence,
taken from \cite[Theorem~3.1]{BenedettoDifferentiationAndBLT}:
\begin{equation}
  \forall \, g \in L^2(\R): \quad
      g \text{ is a Bessel vector } \Longleftrightarrow \, Z g \in L^\infty([0,1]^2).
  \label{eq:BesselSequenceZakCharacterization}
\end{equation}

\subsection{Properties of \texorpdfstring{$\HH^1$}{ℍ¹}}
\label{sec:FatH1Properties}

A further important property that we will use is the following characterization
of the space $\HH^1$ via the Zak transform, a proof of which is given in
\mbox{\cite[Lemma~2.4]{QuantitativeSubspaceBL}}.
\begin{equation}
  \forall \, f \in L^2(\R): \quad
  f \in \HH^1 \, \Longleftrightarrow \, Z f \in W^{1,2}_{{\rm loc}}(\R^2) .
  \label{eq:FatH1ZakCharacterization}
\end{equation}
It is crucial to observe that the Sobolev space $W^{1,2}(\R^2)$
belongs to the ``borderline'' case of the Sobolev embedding theorem,
meaning $\strut W^{1,2}_{{\rm loc}}(\R^2) \not\hookrightarrow L^\infty_{{\rm loc}}(\R^2)$.
In fact, it is easy to verify (see e.g.\ \cite[Page~280]{EvansPDE})
for $x_0 := (\frac{1}{2}, \frac{1}{2})^T \in \R^2$ that the function
\[
  u_0 : \quad
  (0,1)^2 \to \R, \quad
  x \mapsto \ln \left(\ln \left(1 + \frac{1}{|x - x_0|}\right)\right)
\]
belongs to $W^{1,2}( (0,1)^2 )$, but is not essentially bounded.
Now, using the chain rule and the product rule for Sobolev functions
(see e.g.\ \cite[Exercise~11.51(i)]{LeoniSobolev} and \cite[Theorem~1 in Section~5.2.3]{EvansPDE}),
we see that if $\varphi \in C_c^\infty ( (0,1)^2 )$ is chosen such that $0 \leq \varphi \leq 1$
and such that $\varphi \equiv 1$ on a neighborhood of $x_0$, then the function
\begin{equation}
  u : \quad
  \R^2 \to [0,\infty), \quad
  x \mapsto \varphi(x) \cdot \bigl(1 + \sin(u_0(x))\bigr)
  \label{eq:NonContinuousSobolevFunction}
\end{equation}
satisfies $u \in W^{1,2}(\R^2)$, is continuous and bounded on $\R^2 \setminus \{ x_0 \}$,
but $\lim_{x \to x_0} u(x)$ does not exist; this uses that
$\lim_{x \to \infty} \sin(x)$ does not exist and that on each small ball
$B_\eps (x_0)$, the function $u_0$ attains all values from $(M,\infty)$, for a suitable
$M = M(\eps) > 0$.

\subsection{A connection to Fourier series}
\label{sec:FourierSeriesConnection}

In this subsection, we show that for any fixed window $g \in L^2(\R)$
the unconditional convergence of Janssen's representation
in the strong operator topology implies that the partial sums of a certain Fourier series
are uniformly bounded in $L^\infty$.
This connection will be used in the next subsection to disprove the unconditional convergence
of Janssen's representation in the strong operator topology.

Precisely, define $Q := [0,1]^2$.
For $H \in L^\infty(Q)$, define the associated multiplication operator as
\[
  M_H : \quad
  L^2(Q) \to L^2(Q), \quad
  F \mapsto F \cdot H .
\]
It is well-known that $\| M_H \|_{L^2 \to L^2} = \| H \|_{L^\infty}$.

Let us fix any window $g \in L^2(\R)$.
Given a finite set $I \subset \Z^2$, we define
\[
  S_I : \quad
  L^2(\R) \to L^2(\R), \quad
  f \mapsto \sum_{(k,\ell) \in I}
              \langle g, T_k M_\ell g \rangle T_k M_\ell f . 
\]
Using \Cref{eq:ZakDiagonalizesTF} and the isometry of the Zak transform, we then see
\begin{equation*}
  \begin{split}
    Z(S_I f)
    & = \sum_{(k,\ell) \in I}
          \langle Z g, Z[T_k M_\ell g] \rangle_{L^2(Q)} Z[T_k M_\ell f] 
     = Z f \cdot \sum_{(k,\ell) \in I}
                    \langle Z g, Z g \cdot e_{\ell,-k} \rangle_{L^2(Q)}
                    \cdot e_{\ell,-k} \\
    & = Z f \cdot \sum_{(k,\ell) \in I}
                    \langle Z g \cdot \overline{Z g}, e_{\ell,-k} \rangle_{L^2(Q)}
                    \cdot e_{\ell,-k} 
     = Z f \cdot \sum_{(k,\ell) \in I}
                    \widehat{|Z g|^2}(\ell, -k) \cdot e_{\ell,-k} \\          
		&= M_{\Fourier_{I'} [|Z g|^2]} [Z f] ,
  \end{split}
\end{equation*}
where $I' := \{ (\ell,-k) \colon (k,\ell) \in I \}$ and
\[
  \Fourier_J H
  := \sum_{\alpha \in J}
       \widehat{H}(\alpha) \, e_\alpha
  \quad \text{with} \quad
  \widehat{H}(\alpha) = \langle H, e_\alpha \rangle_{L^2(Q)}
  \quad \text{for} \quad
  J \subset \Z^2 .
\]
In other words, we have
\begin{equation}\label{eq:MultiplicationOperatorConnection}
  S_I = Z^{-1} \circ M_{\Fourier_{I'}[|Z g|^2]} \circ Z.
\end{equation}
Given $J \subset \Z^2$, define $J_\ast := \{ (-\ell,k) \colon (k,\ell) \in J \}$
and note $(J_\ast)' = J$.
Now, suppose that $(S_I)_I$ converges strongly to some (bounded) operator, as $I \to \Z^2$;
this is always the case if Janssen's representation converges unconditionally
(to $S$ or some other operator) in the SOT.
Then, given any sequence $(J_n)_{n \in \N}$ of finite subsets $J_n \subset \Z^2$
with $J_n \subset J_{n+1}$ and $\bigcup_{n=1}^\infty J_n = \Z^2$, we see $(J_n)_\ast \to \Z^2$
so that the sequence $(S_{(J_n)_\ast})_{n \in \N}$ converges strongly to some bounded operator.
By the uniform boundedness principle, this shows $\| S_{(J_n)_\ast} \|_{L^2 \to L^2} \leq C$
for all $n \in \N$ and some $C > 0$.
By \Cref{eq:ZakIsUnitary,eq:MultiplicationOperatorConnection},
and because of $((J_n)_\ast)' = J_n$, this implies
\[
  \big\| \Fourier_{J_n} [|Z g|^2] \big\|_{L^\infty (Q)}
  = \big\| M_{\Fourier_{J_n}[|Z g|^2]} \big\|_{L^2 \to L^2}
  \leq C
  \qquad \forall \, n \in \N ,
\]
meaning that the partial Fourier sums $\Fourier_{J_n} [|Z g|^2]$
of the function $|Zg|^2$ are uniformly bounded in $L^\infty(Q)$.

\subsection{The counterexample}\label{sec:MainResult}

In this subsection, we prove the following:

\begin{prop}\label{prop:MainResult}
  There exists a Bessel vector $g \in \HH^1(\R)$ such that the series
  defining Janssen's representation of the frame operator $S = S_g = S_{\Z \times \Z, g, g}$
  associated to $g$ is \emph{not} unconditionally convergent in the strong operator topology.
\end{prop}

To prove the proposition, we consider the function $F := u : (0,1)^2 \to [0,\infty)$
introduced in \Cref{eq:NonContinuousSobolevFunction}.
The properties of $F$ that we need are the following:
\begin{enumerate}
  \item $F$ has compact support in $(0,1)^2$,
        say $\supp F \subset (\delta, 1-\delta)^2$
        for some $\delta \in (0,\frac{1}{2})$.
  \item $F$ is bounded, but discontinuous at $x_0 \in (0,1)^2$
        (even after adjusting $F$ on a set of measure zero).
  \item $F \in W^{1,2}\bigl( (0,1)^2 \bigr)$.
\end{enumerate}
We now extend $F$ by zero to $[0,1)^2$ and then extend $1$-periodically in both coordinates to $\R^2$. 
Thanks to the compact support of $F$, it is easy to see $F \in W^{1,2}_{{\rm loc}}(\R^2)$.

Furthermore, we consider the function
\[
  G_0 : \quad
  \R^2 \to \CC, \quad
  (x,\omega) \mapsto e^{2 \pi i \lfloor x \rfloor \omega} ,  
\]
where for each $x\in\R$, $\lfloor x \rfloor \in \Z$
denotes the unique integer such that $x \in \lfloor x \rfloor + [0,1)$.
It is then straightforward to verify that $G_0$ is quasi-periodic
(see \Cref{eq:QuasiPeriodicity}), i.e., $G_0(x+m,\omega) = e^{2 \pi i m \omega} G_0(x,\omega)$
and $G_0(x,\omega+m) = G_0(x,\omega)$ for $x,\omega \in \R$ and $m \in \Z$.
Since $F$ is $1$-periodic in both coordinates,
it is easy to see that $F \cdot G_0$ is quasi-periodic as well.

Finally, we choose a smooth function $\psi : \R \to \R$ satisfying
$\psi(x) = n$ for all $x \in n + [\delta,1-\delta]$ with $n \in \Z$,
and define
\[
  G : \quad
  \R^2 \to \CC, \quad
  (x,\omega) \mapsto e^{2 \pi i \psi(x) \omega} .
\]
Using that $F(x,\omega) = 0$ for $n \in \Z$ and $x \in [n,n+1] \setminus (n +\delta,n +1-\delta)$,
it is easy to check $F \cdot G_0 = F \cdot G$, so that
$H := F \cdot G \in W^{1,2}_{{\rm loc}}(\R^2) \subset L^2_{{\rm loc}}(\R^2)$
is quasi-periodic.

Since the Zak transform $Z : L^2(\R) \to L^2( (0,1)^2 )$ is unitary,
there exists a unique function $g \in L^2(\R)$ such that $(Z g)|_{(0,1)^2} = H|_{(0,1)^2}$.
Since both $Z g$ and $H$ are quasi-periodic, this implies $Z g = H$ almost everywhere.
Since $H \in W^{1,2}_{{\rm loc}}(\R^2)$ is bounded,
\Cref{eq:BesselSequenceZakCharacterization,eq:FatH1ZakCharacterization} show that
$g \in \HH^1$ is a Bessel vector.
Let us assume towards a contradiction that Janssen's representation of the frame operator
associated to $g$ converges unconditionally in the strong operator topology.

Note that $|Z g|^2 = |H|^2 = F^2$ is discontinuous at $x_0 \in (0,1)^2$
(since $F$ is discontinuous there and also non-negative),
even after possibly changing $|Z g|^2$ on a null-set.
In particular, this implies that the Fourier coefficients $c_\alpha := \widehat{|Z g|^2}(\alpha)$
(for $\alpha \in \Z^2$) satisfy $c = (c_\alpha)_{\alpha \in \Z^2} \notin \ell^1(\Z^2)$,
since otherwise the Fourier series of $|Z g|^2$ would be uniformly convergent.
This implies $\sum_{\alpha \in \Z^2} |\Re c_\alpha| = \infty$
or $\sum_{\alpha \in \Z^2} |\Im c_\alpha| = \infty$.
For simplicity, we assume the first case; the second case can be treated by similar arguments.
This implies that there exists an enumeration $(\alpha_n)_{n \in \N}$ of $\Z^2$
such that $|\sum_{n=1}^N \Re c_{\alpha_n}| \to \infty$ as $N \to \infty$.
Indeed, if $\sum_{\alpha \in \Z^2} (\Re c_\alpha)_+ < \infty$
or $\sum_{\alpha \in \Z^2} (\Re c_\alpha)_- < \infty$, this is trivial
(for every enumeration); otherwise, existence of the desired
enumeration follows from the Riemann rearrangement theorem
(see e.g., \cite[Theorem~3.54]{RudinPrinciplesOfAnalysis}).

Now, define $J_n := \{ \alpha_1,\dots,\alpha_n \}$ for $n \in \N$.
We have seen in \Cref{sec:FourierSeriesConnection} that the partial Fourier sums
$\Fourier_{J_n} [|Z g|^2]$ are uniformly bounded in $L^\infty$,
say $\| \Fourier_{J_n} [|Z g|^2] \|_{L^\infty} \leq C$ for all $n \in \N$.
Since each $\Fourier_{J_n} [|Z g|^2]$ is continuous (in fact, a trigonometric polynomial),
this implies
\begin{align*}
  C
  \geq \big|\Fourier_{J_N} |Z g|^2 (0) \big|
  & =    \Big|
           \sum_{\alpha \in J_N}
             \widehat{|Z g|^2}(\alpha)
             \cdot e^{2 \pi i \langle \alpha,0 \rangle}
         \Big| \\
  & =    \Big|
           \sum_{\alpha \in J_N}
             c_\alpha
         \Big|
    \geq \Big| \Re \sum_{n=1}^N c_{\alpha_n}  \Big|
    \to \infty \text{ as } N \to \infty ,
\end{align*}
which is the desired contradiction.

\subsection{Conditional divergence of Janssen's representation in the operator norm}
\label{sec:OperatorNormNonConvergence}

We showed above that \emph{unconditional} convergence of Janssen's representation
\eqref{eq:JannssenRepresentation} in the strong operator topology fails
for some Bessel vector $g \in \HH^1(\R)$.
A similar argument shows that convergence in the operator norm
(with respect to \emph{any} given enumeration) also fails in general:
Using \Cref{eq:MultiplicationOperatorConnection}
(or more generally the arguments in \Cref{sec:FourierSeriesConnection}),
it is relatively easy to see that  if
for some enumeration $\Z^2 = \{ \alpha_n \colon n \in \N \}$
and $I_n := \{ \alpha_1,\dots,\alpha_n \}$, the sequence of partial sums $(S_{I_n})_{n \in \N}$
of Janssen's representation \eqref{eq:JannssenRepresentation} converges in operator norm
(not even necessarily to $S$),
then the associated sequence $(\Fourier_{I_n'} [|Z g|^2])_{n \in \N}$
of partial Fourier sums of $|Z g|^2$ is Cauchy in $L^\infty(Q)$ and thus
converges uniformly on $Q$ to a (necessarily continuous) function $\widetilde{H} : Q \to \CC$.
However, since $|Z g|^2 = |H|^2 = F^2 \in L^\infty(Q) \subset L^2(Q)$,
we know that $\Fourier_{I_n '} [|Z g|^2] \to F^2$ with convergence in $L^2(Q)$.
Hence, $F^2 = \widetilde{H}$ almost everywhere on $Q$, where $\widetilde{H}$ is continuous.
But we saw above that $F^2$ is discontinuous on $Q$, even after (possibly) changing it on a null-set.
Thus, we have obtained the desired contradiction:

\begin{prop}
There are Bessel vectors $g \in \HH^1$ for which Janssen's representation fails to converge conditionally in the operator norm.
\end{prop}


However, we leave it as an open question whether Janssen's representation converges conditionally in the strong sense for Bessel vectors $g \in \HH^1$.


\section*{Acknowledgments}
D.G.\ Lee acknowledges support by the DFG Grants PF 450/6-1 and PF 450/9-1. 
F.\ Philipp was funded by the Carl Zeiss Foundation within the project \textit{DeepTurb---Deep Learning in und von Turbulenz}.
F.\ Voigtlaender acknowledges support by the German Science Foundation (DFG)
in the context of the Emmy Noether junior research group VO 2594/1-1.

\section*{Author Affiliations}
\end{document}